\documentclass[10pt]{amsart}

\usepackage[all]{xypic}

\usepackage{latexsym}

\usepackage{amssymb}
\usepackage{amsfonts}
\usepackage{amscd}
\usepackage{amsmath,amsthm}
\usepackage{soul}

\newtheorem{lemma}{Lemma}[section]

\newtheorem{theorem}[lemma]{Theorem}

\numberwithin{equation}{section}

\def\NN{{\mathbb N}}

\def\Hom{\operatorname {Hom}}

\def\1I{\operatorname {I}}
\def\2I{\operatorname {II}}

\def\Tor{\operatorname {Tor}}

\def\dim{\operatorname{dim}}

\def\Fdim{{\sf Fdim}}

\def\Gr{{\sf Gr}}

\def\Hom{\operatorname{Hom}}

\def\id{\operatorname{id}}

\def\liminj{\varinjlim}

\def\Projnc{\operatorname{Proj}_{nc}}

\def\QGr{\operatorname{\sf QGr}}

\def\Spec{\operatorname{Spec}}

\def\G{\mathop{\underline{\underline{\it \Gamma}}}\nolimits}

\def\d{\downarrow}

\def\a{\alpha}
\def\b{\beta}
\def\c{\gamma}
\def\d{\delta}

\def\g{\gamma}

\def\ve{\varepsilon}

\def\G{\Gamma}

\def\sA{{\sf A}}
\def\sB{{\sf B}}

\def\sS{{\sf S}}
\def\sT{{\sf T}}

\def\cL{{\mathcal L}}
\def\cM{{\mathcal M}}
\def\cN{{\mathcal N}}
\def\cO{{\mathcal O}}

\def\Qcoh{{\sf Qcoh}}

\newdimen\uboxsep \uboxsep=1ex
\def\uboxn#1{\vtop to 0pt{\hrule height 0pt depth 0pt\vskip\uboxsep
\hbox to 0pt{\hss #1\hss}\vss}}

\def\uboxs#1{\vbox to 0pt{\vss\hbox to 0pt{\hss #1\hss}
\vskip\uboxsep\hrule height 0pt depth 0pt}}

\def\hoek{\hbox{\vtop{\hbox{\vrule\phantom{xx}\vrule}\hrule}\kern -0.4pt}}

\def\dirlim{\mathop{\vtop{\baselineskip -100pt\lineskip -1pt\lineskiplimit 0pt
\setbox0\hbox{lim}\copy0\hbox to \wd0{\rightarrowfill}}}\limits}
\def\invlim{\mathop{\vtop{\baselineskip -100pt\lineskip -1pt\lineskiplimit 0pt
\setbox0\hbox{lim}\copy0\hbox to \wd0{\leftarrowfill}}}\limits}

\def\I11{{1 \kern -0.8pt \! \mbox{l}}}
\def\mumu{{\mu\kern-4.2pt\mu}}

\def\boxtimes{\setbox0\hbox{$\Box$}\copy0\kern-\wd0\hbox{$\times$}}

\begin{document}

\title[Corrigendum: Maps between non-commutative spaces]{Corrigendum to  
``Maps between non-commutative spaces''  
[Trans. Amer. Math. Soc.,
356(7) (2004) 2927--2944]}

\author{S. Paul Smith}

\address{ Department of Mathematics, Box 354350, Univ.
Washington, Seattle, WA 98195}

\email{smith@math.washington.edu}



\begin{abstract}
The statement of Lemma 3.1 in the published paper is not correct. Lemma 3.1 is needed for the proof of 
Theorem 3.2. Theorem 3.2 as originally stated is  true but its ``proof'' is not correct. 
Here we change the statements and proofs of Lemma 3.1 and Theorem 3.2. 
We also prove a new result. Let $k$ be a field, $A$ a left and right noetherian 
$\NN$-graded $k$-algebra such that $\dim_k(A_n)< \infty$ for all $n$, and $J$ a graded two-sided ideal of $A$. 
If the non-commutative scheme $\Projnc(A)$ is isomorphic to a projective scheme $X$, then there is a closed subscheme
$Z \subseteq X$ such that  $\Projnc(A/J)$ is isomorphic to $Z$. This result is a geometric translation of what we actually prove: if the category $\QGr(A)$ is equivalent to $\Qcoh(X)$, then $\QGr(A/J)$ is equivalent to $\Qcoh(Z)$ for
some closed subscheme $Z \subseteq X$.
\end{abstract}

\maketitle

\pagenumbering{arabic}

\section{Replacements for \cite[Lem. 3.1 and Thm. 3.2]{S}}

Lemma 3.1 and its proof in \cite{S} should be replaced by Lemma \ref{lem.quotient.adjoint} below.

Lemma \ref{lem.quotient.adjoint} is inspired by a result of P. Gabriel  \cite[Cor. 2, p. 368]{Gab} which says that under suitable conditions, when the functor $F$ in (\ref{comm.diag}) is 
exact and $F(\sS) \subseteq \sT$,  then there is a unique functor $G:\sA/\sS \to \sB/\sT$
that makes  (\ref{comm.diag}) commute. 
 
We assume that all our categories have small Hom sets.

\begin{lemma}
\label{lem.quotient.adjoint}
Let $\sA$ and $\sB$ be Grothendieck categories with localizing
subcategories $\sS \subseteq \sA$ and $\sT \subseteq \sB$.
Let $\pi:\sA \to \sA/\sS$ and
$\pi':\sB \to \sB/\sT$ be the quotient functors, and let $\omega$
and $\omega'$ be right adjoints to $\pi$ and $\pi'$ respectively. 
Consider the following diagram of functors:
\begin{equation}
\label{comm.diag}
\UseComputerModernTips
\xymatrix{
\sA \ar[r]^F  \ar[d]_{\pi}  &   \sB \ar[d]^{\pi'}
\\
\sA/\sS & \sB/\sT.
}
\end{equation}
Suppose that $F$ is right exact and $F(\sS) \subseteq \sT$. 
If $L_1(\pi'F)$, the first left derived functor of $\pi'F$, exists, which it does when $\sA$ has enough projectives, 
and vanishes on $\sS$, then 
\begin{enumerate}
  \item 
  the natural transformation $\pi'F \to \pi'F\omega\pi$ induced by the natural transformation
$\eta:\id_{\sA} \to \omega\pi$ is an isomorphism;
  \item 
  there is a unique functor $G:\sA/\sS \to \sB/\sT$ such that
$\pi'F=G\pi$; furthermore, $G\cong \pi'F\omega$;
  \item 
  $G$ is right exact;
  \item{}
 if $F$ commutes with direct sums so does $G$;
\item{}
if $F$ has a right adjoint so does $G$: if $F\dashv H$,  then $G \dashv \pi H\omega'$.  
\end{enumerate} 
\end{lemma}
\begin{proof}
{\bf (1)}
Let $a \in \sA$. Let $\G:\sA \to \sS$ be the functor that  sends an object in $\sA$ to the sum of its subobjects that belong to $\sS$.
The exact sequence
$$
0 \longrightarrow \G a \longrightarrow a \stackrel{\eta_a}{\longrightarrow} \omega\pi a \longrightarrow (R^1\G)a \longrightarrow 0
$$
is a ``composition'' of exact sequences 
\begin{align*}
& 0 \longrightarrow \G a \longrightarrow a \stackrel{\a}{\longrightarrow} x \longrightarrow 0 \qquad  \qquad \phantom{xxx} \hbox{and}
\\
& 0 \longrightarrow x  \stackrel{\b}{\longrightarrow} \omega\pi a \longrightarrow (R^1\G)a \longrightarrow 0
\end{align*}
in which $\b\a=\eta_a$. Applying $\pi'F$ to these sequences gives  exact sequences 
$$
\xymatrix{
 \pi'F( \G a) \ar[r]  & \pi'Fa \ar[rr]^{(\pi'F)\a} &&  \pi'Fx \ar[r] & 0  
}
$$
and
$$
\xymatrix{
 \big(L_1(\pi'F)\big)\big( (R^1\G)a\big) \ar[r]&  \pi'Fx   \ar[rr]^{(\pi'F)\b} &&  \pi' F\omega\pi a \ar[r] &  \pi'F\big((R^1\G)a\big) \ar[r] & 0.
}
$$
Since $\pi'F$ and $L_1(\pi'F)$ vanish on $\sS$,  $\pi' F( \G a) = \big(L_1(\pi'F)\big)\big( (R^1\G)a\big) =  \pi'F\big((R^1\G)a\big) =0$.  Hence $(\pi'F)\a$ and $(\pi'F)\b$ are isomorphisms. But $(\pi'F)\b \circ (\pi'F)\a=(\pi'F)(\b \circ \a)=(\pi'F)\eta_a$, so $(\pi'F)\eta_a$ is an isomorphism.

{\bf (2)}
 We will now show there is a unique functor $G:\sA/\sS \to \sB/\sT$ such that $\pi'F=G\pi$. If $\pi'F$ were exact, the existence of 
 a unique such $G$ is proved by Gabriel \cite[Cor.2, p.368]{Gab}. We can't appeal to op. cit. because we only know that 
 $\pi'F$ is right exact (and $L_1(\pi'F)$ vanishes on $\sS$). Nevertheless, the proof at loc. cit. can be modified to establish the 
 existence and uniqueness of $G$ such that $\pi'F=G\pi$. (Our notation differs from Gabriel's. Our $\pi'F$ plays the role of his
 $G$ and our $G$ plays the role of his $H$.)
 
 Let $m$ and $n$ be objects in $\sA$. Let $i:m' \to m$ be a subobject and $p:n \to n/n'$ a quotient object such that 
 $m/m'$ and $n'$ belong to $\sS$. The maps  $\varphi$ and $\psi$ defined by
 $\varphi(\a):=p\a i$ and $\psi(\b):=\pi'F(p) \circ \b \circ \pi'F(i)$ fit into a commutative diagram
 \begin{equation}
 \label{defn.psi}
 \xymatrix{
\ar[d]_\varphi  \Hom_{\sA}(m,n) \ar[rr]^{\pi'F} &&  \Hom_{\sB/\sT}(\pi'Fm,\pi'Fn) \ar[d]^{\psi}
\\ 
 \Hom_{\sA}(m',n/n') \ar[rr]_{\pi'F} && \Hom_{\sB/\sT}(\pi'Fm',\pi'F(n/n')).
 }
 \end{equation}
 
In order to show that $\psi$ is surjective, suppose that $\d \in  \Hom_{\sB/\sT}(\pi'Fm',\pi'F(n/n'))$. The rows in the following
diagram are exact:
$$
\xymatrix{
0 = \pi'F(n') \ar[r] & \pi'F(n) \ar[rr]^{\pi'F(p)} && \pi'F(n/n') \ar[r] & 0
\\
 0 = L_1(\pi'F)(m') \ar[r] & \pi'F(m') \ar[urr]^\d \ar[rr]_{\pi'F(i)}  && \pi'F(m) \ar[r] & \pi'F(m/m') = 0.
 }
 $$
 Because the arrows $\pi'F(n) \to \pi'F(n/n')$ and $ \pi'F(m') \to \pi'F(m)$ are isomorphisms there is a unique morphism 
 $\b:\pi'F(m) \to \pi'F(n)$ such that 
 $$
\xymatrix{
 \pi'F(n) \ar[rr]^{\pi'F(p)} && \pi'F(n/n')  
\\
\ar[u]^\b  \pi'F(m) && \pi'F(m') \ar[ll]^{(\pi'F)(i)}.  \ar[u]_\d 
 }
 $$
In other words, $\psi(\b)=\d$. Thus, $\psi$ is surjective. The uniqueness of $\b$ implies that $\psi$ is also injective,
and therefore bijective. 

As $m$ and $n$ run over all objects in $\sA$, 
the maps 
$$
\psi^{-1} \circ \pi'F: \Hom_{\sA}(m',n/n') \longrightarrow \Hom_{\sB/\sT}(\pi'Fm,\pi'Fn),
$$
as $m'$ and $n'$ take all possible values, determine maps
$$
G: \Hom_{\sA/\sS}(\pi m, \pi n) =  \liminj_{m',n'} \Hom_{\sA}(m',n/n') \longrightarrow \Hom_{\sB/\sT}(\pi'Fm,\pi'Fn)
$$
for all objects $\pi m$ and $\pi n$ in $\sA/\sS$. If we define $G(\pi m)=\pi'F(m)$ for all objects $\pi m \in \sA/\sS$, then $G$ becomes
a functor $\sA/\sS \to \sB/\sT$. By definition, $G\pi=\pi' F$ on objects in $\sA$. 

Suppose $\a:m \to n$ is a morphism in $\sA$. Let $m'=m$ and $n'=0$. The maps $i$ and $p$ are the identity morphisms
on $m$ and $n$ so the map $\psi$ in (\ref{defn.psi}) is the identity; hence $G \pi(\a)=\psi^{-1}\pi'F(\a)=\pi'F(\a)$. 
The functors $G\pi$ and $\pi'F$ are therefore equal. 

To prove the uniqueness of $G$, suppose that  $G':\sA/\sS \to \sB/\sT$ is ``another'' functor such that  $G'\pi=\pi' F$.
If $x$ is an object in $\sA/\sS$, then $x=\pi m$ for some object $m$ in $\sA$ so $G'(x)=G'\pi(m)=\pi F(m)=G\pi(m)=G(x)$.
If $\theta$ is a morphism in $\sA/\sS$, then $\theta=\pi(\a)$ for some morphism $\a$ in $\sA$ so
$G'(\theta)=G'\pi(\a)=\pi' F(\a)=G\pi(\a)=G(\theta)$. Hence $G'=G$. 

Finally, since $G\pi=\pi'F$, $\pi' F \omega = G \pi\omega \cong G$.

{\bf (3)}
By \cite[Cor. 1, p.368]{Gab}, every short exact sequence in $\sA/\sS$ is isomorphic to one of the form
\begin{equation}
\label{eq.G.rt.exact}
  \xymatrix{
  0 \ar[r] & \pi a \ar[rr]^{\pi(\a)} && \pi b \ar[rr]^{\pi(\b)} && \pi c \ar[r] & 0
  }
\end{equation}
where   
$
  \xymatrix{
  0 \ar[r] &  a \ar[r]^\a & b \ar[r]^\b &  c \ar[r] & 0
  }
$
is an exact sequence in $\sA$. Because $\pi'F$ is right exact, 
$$
  \xymatrix{
   \pi'Fa \ar[rr]^{(\pi'F)(\a)} && \pi'Fb \ar[rr]^{(\pi'F)\b} &&  \pi'Fc \ar[r] & 0
  }
$$
is exact.
There is a commutative diagram
$$
  \xymatrix{
   \pi'Fa \ar[d]_{(\pi'F)(\eta_a)}  \ar[rr]^{(\pi'F)\a}&& \pi'Fb \ar[rr]^{(\pi'F(\b)}  \ar[d]^{(\pi'F)(\eta_b)}  &&  \pi'Fc   \ar[d]^{(\pi'F)(\eta_c)}  \ar[r] & 0
   \\
  \pi'F\omega \pi a \ar@{=}[d]  \ar[rr]_{(\pi'F\omega \pi)( \a)}&& \pi'F\omega \pi  b  \ar@{=}[d]  \ar[rr]_{(\pi'F\omega \pi)( \b)}  &&  \pi'F\omega \pi  c   \ar@{=}[d]   \ar[r] & 0
     \\
G(\pi a)   \ar[rr]_{G(\pi \a)}&& G( \pi  b)  \ \ar[rr]_{G(\pi \b)}  &&  G(\pi  c)    \ar[r] & 0
  }
$$
in which, by (1),  the upper-most vertical arrows provide an isomorphism from the top row to the middle row. 
The bottom row in the diagram is therefore exact. But the bottom row is 
obtained by applying $G$ to (\ref{eq.G.rt.exact}) so $G$ is, as claimed, right exact.

{\bf (4)}
Let $x_i$, $i \in I$, be a collection of objects in $\sA/\sT$. There are objects $a_i \in \sA$ such that $x_i=\pi(a_i)$ for all $i$.
Thus
$$
G\bigg(\bigoplus_{i \in I} x_i\bigg) = G\bigg(\bigoplus_{i \in I} \pi a_i\bigg) \cong G\pi\bigg(\bigoplus_{i \in I} a_i\bigg) 
=
\pi'F\bigg(\bigoplus_{i \in I} a_i\bigg) = \bigoplus_{i \in I} \pi'Fa_i.
$$ 
But $\pi'F=G\pi$ so  
$$
 \bigoplus_{i \in I} \pi'Fa_i =  \bigoplus_{i \in I} G\pi a_i =  \bigoplus_{i \in I} Gx_i.
$$
Therefore $G$ commutes with direct sums. 

{\bf (5)}
Suppose $F \dashv H$. Because $F$ has a right adjoint it commutes with direct sums.
By (4), $G$ also commutes with direct sums.  Since $\sA$ is a Grothendieck category so is $\sA/\sS$. In particular, $\sA/\sS$
is cocomplete, has a generator, and has small Hom sets. 
Thus, by the Special Adjoint Functor Theorem, the dual of \cite[Cor.V.8, p.130]{Mac}, $G$ has a right adjoint, $G'$ say. Since $\pi \dashv \omega$  and $G \dashv G'$, $G \pi\dashv \omega G'$.
But $G\pi =\pi'F$ so $\pi'F\dashv \omega G'$; but $H\omega'$ is also right adjoint to $\pi'F$ so 
$\omega G' \cong H\omega'$ whence  $G' \cong \pi\omega G' \cong \pi H \omega'$. 
 Thus, $G \dashv \pi H \omega'$, as claimed.
\end{proof}

{\bf Remark.}
With regard to Lemma \ref{lem.quotient.adjoint}(2), 
the hypothesis that $L_1(\pi'F)$ vanishes on $\sS$ is not only sufficient to prove there is a $G$ such that $G\pi=\pi'F$;
it is also necessary. If there is a $G$ such that $G\pi=\pi'F$, then $\pi'F(\sS)=0$ because $\pi(\sS)=0$ and, because $\pi$ is exact, $L_1(\pi'F) =  L_1(G\pi) \cong (L_1G) \circ \pi$ which shows that $L_1(\pi'F)$ vanishes on $\sS$. 

\bigskip

From now on $k$ denotes an arbitrary field.
 
In the next result $A$ is a right and left noetherian $\NN$-graded $k$-algebra such
that $\dim_k(A_n)< \infty$ for all $n$ and $J$ is a graded ideal in $A$. 

We write $\Gr(A)$ for the category of graded right $A$-modules, $\Fdim(A)$ for the full subcategory of $\Gr(A)$
consisting of modules that are the sum of their finite-dimensional submodules, and $\QGr(A)$ for the quotient category 
$\Gr(A)/\Fdim(A)$.

The inclusion functor $f_*:\Gr(A/J) \to \Gr(A)$ has a left adjoint $f^*$ and a right adjoint $f^!$.

Theorem 3.2 and its proof in \cite{S} should be replaced by the following result.

\begin{theorem}
\label{thm.closed.proj}
\label{thm.3.2}
Let $J$ be a graded ideal in an $\NN$-graded $k$-algebra $A$ and consider the diagram
$$
\xymatrix{
\Gr (A/J) \ar[r]^{f_*}  \ar[d]^{\pi'}  & \Gr (A)  \ar[d]_{\pi}
\\ 
\QGr(A/J) \ar@/^1pc/[u]^{\omega'}  &  \QGr(A)  \ar@/_1pc/[u]_{\omega}
}
$$
in which $\pi$ and $\pi'$ denote the quotient functors and $\omega$ and $\omega'$ are their right adjoints.
Define $i^*= \pi' f^*\omega$ and $i^!= \pi' f^!\omega$. Then 
\begin{enumerate}
\item{}
there is a unique functor $i_*:\QGr(A/J) \to \QGr(A)$ such that $i_* \pi' = \pi f_*$; furthermore, $i_*$ is exact;
\item{}
 $i_* \cong \pi f_*\omega'$;
   \item 
  $i_*$ is fully faithful;
  \item 
  $i^*$ is left adjoint and $i^!$ right adjoint to $i_*$; 
  \item 
  the essential image of $i_*$ is closed under quotients and submodules.
\end{enumerate}
In the language of \cite{S},  $\Projnc (A/J)$ is a closed subspace of $\Projnc(A)$ and the homomorphism $A \to A/J$ induces a closed immersion 
$i:\Projnc (A/J) \to \Projnc( A)$. 
\end{theorem}
\begin{proof}
{\bf (1)}
This is due to P. Gabriel  \cite[Cor. 2, p. 368]{Gab}  and \cite[Cor. 3, p. 369]{Gab}. In particular, the diagram
\begin{equation}
\label{diag.i_*.comms}
\xymatrix{
\Gr (A/J) \ar[r]^{f_*}  \ar[d]_{\pi'}  & \Gr (A)  \ar[d]^{\pi}
\\
\QGr(A/J) \ar[r]_{i_*}  &  \QGr(A)
}
\end{equation}
commutes.  

{\bf(2)}
Let $\eta':\id_{\Gr(A/J)} \to \omega'\pi'$ and $\ve':\pi'\omega' \to \id_{\QGr(A/J)}$ be the unit and count 
associated to the adjoint pair $\pi'\dashv \omega'$. Since $\ve'$ is an isomorphism, by \cite[Prop. 3(a), p. 371]{Gab}, the natural transformation $i_* \to i_*\pi'\omega'$
it induces is also an isomorphism. But $i_*\pi'\omega'=\pi f_* \omega'$ so $i_* \cong \pi f_* \omega'$.

{\bf (3)}
 We need two preliminary results before completing  the proof of (3). 

{\bf (3a)}: If $T \in \Fdim(A)$, then $T \otimes_A J \in \Fdim(A)$.

\underline{Proof}: 
Since $A$ is left noetherian, $J=Ax_1+\cdots+Ax_s$ for some homogeneous elements $x_i \in J$. There is an integer
$d$ such that $\deg(x_i) \le d$ for all $i=1,\ldots,s$.  

Let $t \in T$ and $x \in J_m$.  There is an integer $\ell$ such that $tA_{\ge \ell}=0$. 
 
Let $a \in A_{\ge  \ell+d-m}$. Then $xa=a_1x_1+\cdots+a_sx_s$ for some elements 
$a_i \in A_{\deg(xa)-\deg(x_i)}\subseteq A_{\ge  \ell}$. Since $ta_1=\cdots =ta_s=0$,
$(t \otimes x)a=ta_1 \otimes x_1+\cdots + ta_s \otimes x_s =0$. We have shown that $(t \otimes x)A_{\ge \ell+d-m}=0$.
It follows that every element in $T \otimes _AJ$ generates a finite dimensional module, i.e., $T \otimes_A J \in \Fdim(A)$.
 $\lozenge$


{\bf (3b)}: If  $M \in \Gr( A/J)$, then $\omega\pi f_*M \cong f_*\omega'\pi'M$.

\underbar{Proof}: 
Let  $\overline{M}$ denote the image of $M$ under the natural map $M \to \omega\pi M$. 
 By definition, $\omega\pi M$ is ``the'' largest essential extension
$0 \to \overline{M} \to \omega \pi M \to T \to 0$  in $\Gr(A)$ such that $T \in \Fdim (A)$.

In the next diagram $\psi$ and $\phi$ are the multiplication maps.
The top row in 
  $$
\UseComputerModernTips
\xymatrix{
\overline{M} \otimes_A J \ar[r]^<<<<<\a  \ar[d]_\psi &  (\omega\pi M) \otimes_A J \ar[r]^<<<<<\b \ar[d]^\phi & T \otimes_A J \ar[r] & 0
\\
\overline{M}J \ar[r]_\theta & (\omega\pi M)J
}
$$
is exact. Since $\overline{M}J=0$, there is a homomorphism $\gamma:T \otimes_A J \to
(\omega \pi M)J$ such that $\gamma\b=\phi$. Since $\phi$ is surjective so is $\c$.

Since $T \otimes_A J \in \Fdim(A)$ and $\c$ is surjective, $(\omega \pi M)J \in \Fdim(A)$ also. This
implies that $\overline{M} \cap (\omega \pi M)J \in \Fdim(A)$. But the only submodule of $\omega\pi M$ that belongs to $\Fdim(A)$ is the zero submodule, so $\overline{M} \cap (\omega \pi M)J =0$. But $\overline{M}$ is essential in 
$\omega\pi M$ so  $(\omega \pi M)J=0$, i.e.,  $\omega\pi M \in \Gr(A/J)$.

Thus, $\omega\pi M$ is ``the'' largest essential extension of $\overline{M}$ in $\Gr(A/J)$ such that 
$(\omega\pi M)/\overline{M}$ is in $\Fdim (A/J)$.
Hence $\omega \pi M$ and $\omega'\pi' M$ are isomorphic as $A/J$-modules.  More 
precisely, $\omega \pi f_*M \cong f_*\omega'\pi' M$.  This completes the proof.
$\lozenge$

 {\bf (3c)}: 
The functor $i_*$ is  fully faithful. 

\underbar{Proof}: 
Let $\a$ be a non-zero morphism in $\QGr(A/J)$. Then $\pi'\omega'(\a) \ne 0$. The image of $\omega'(\a)$ is therefore
not in $\Fdim(A)$. Hence the image of $f_*\omega'(\a)$ is not in $\Fdim(A)$. It follows that $\pi f_*\omega'(\a) \ne 0$.
But $\pi f_*\omega'=i_*$ so $i_*$ is faithful. 

To see that $i_*$ is full, let  $\cM,\cN \in \QGr(A/J)$ and suppose that $\a:i_*\cM \to i_*\cN$ is a morphism. 

There are $A/J$-modules $M$ and $N$ such that $\cM=\pi' M$ and $\cN=\pi'N$. 
Fix isomorphisms $\theta_M:f_*\omega'\pi' M \to \omega \pi f_*M$ and $\theta_N:f_*\omega'\pi' N \to \omega \pi f_*N$.

Then  $\pi(\theta_N)\circ \a \circ \pi(\theta_M)^{-1}$ is in $\Hom_{\QGr(A)}(\pi\omega\pi f_*M,\pi\omega\pi f_*N)$.
By \cite[Lemme 1, p.370]{Gab}, $\pi(\theta_N)\a \pi(\theta_M)^{-1}=\pi(\b)$ for some 
$\b \in \Hom_{\Gr(A)}(\omega\pi f_*M,\omega\pi f_*N)$. Therefore  $\a=\pi(\theta_N^{-1}\b\theta_M)$. Since $f_*$ is full,
the map $\theta_N^{-1}\b\theta_M: f_*\omega'\pi' M \to f_*\omega'\pi' N$ is equal to $f_*(\c)$ for some $\g
\in \Hom_{\Gr(A/J)}(\omega'\pi' M,\omega'\pi' N)$. Hence $\a=\pi f_*(\c)$. Since $\pi f_*=i_*\pi'$, $\a=i_*\pi'(\c)$
which completes the proof that $i_*$ is full. 
$\lozenge$

{\bf (4)} 
We separate the proof into two parts.

{\bf (4a)}
In this part we show that $i^*:=\pi' f^* \omega$ is left adjoint to $i_*$.  

If $N$ is a finite dimensional graded right $A$-module, then $N \otimes_A (A/J)$ is also finite dimensional. 
Since $f^*=- \otimes_A (A/J)$ it commutes with direct limits 
and therefore sends $\Fdim(A)$ to $\Fdim(A/J)$. Hence $\pi'f^*$ vanishes on $\Fdim(A)$.

We wish to apply Lemma \ref{lem.quotient.adjoint} with $F= f^*$, $\sA=\Gr(A)$, and $\sB=\Gr(A/J)$. 
To see that $F=f^*$ satisfies the hypotheses of Lemma \ref{lem.quotient.adjoint} we first note that $f^*$ is right exact 
and sends $\Fdim(A)$ to $\Fdim(A/J)$. The left derived functor $L_1(\pi'f^*)$ exists because $\Gr(A)$ has enough
projectives. We must show that $L_1(\pi'f^*)$  vanishes on $\Fdim(A)$.  
 Since $\pi'$ is exact, $L_1(\pi' f^*) \cong \pi' \circ L_1f^* \cong \pi' \circ \Tor^A_1(-,A/J)$ 
so it suffices to show that $\Tor^A_1(T,A/J)$ is in $\Fdim(A/J)$ whenever $T \in \Fdim(A)$.  
But $\Tor^A_1(T,A/J)$ is isomorphic to a submodule of $T \otimes_A J$ which is, by (3a), in $\Fdim(A)$.

By Lemma \ref{lem.quotient.adjoint}(5), there is a unique functor $G:\QGr(A) \to \QGr(A/J)$ such that 
$\pi'f^*=G\pi$. Furthermore $G \cong  \pi'f^*\omega$. Since $f^* \dashv f_*$,  Lemma \ref{lem.quotient.adjoint}(2)
tells us that $G\dashv \pi f_*\omega'$, i.e., $i_*$ is right adjoint to $i^*$. 

{\bf (4b)}
In this part we show that $i^!:=\pi'f^!\omega$ is right adjoint to $i_*$.
 
To do so we will apply Lemma \ref{lem.quotient.adjoint} to the functor $F= f_*$ with $\sA=\Gr(A/J)$ and $\sB=\Gr(A)$.
To see that $F=f_*$ satisfies the hypotheses of Lemma \ref{lem.quotient.adjoint} we first note that $f_*$ is right exact, in fact exact, and
sends $\Fdim(A/J)$ to $\Fdim(A)$. Hence $\pi f_*$ is right exact, in fact exact, and 
 vanishes on $\Fdim(A/J)$. Since $\pi f_*$ is exact its left derived functor $L_1(\pi f_*)$  
 certainly vanishes on $\Fdim(A/J)$. Lemma \ref{lem.quotient.adjoint} therefore applies. The functor $G$ in part (2) of 
 Lemma \ref{lem.quotient.adjoint} is $i_*$.
  Since $f_* \dashv f^!$, Lemma \ref{lem.quotient.adjoint}(5) tells us that $G \dashv \pi' f^!\omega$, i.e., $i_* \dashv \pi' f^!\omega$.

{\bf (5)}  
We will now show that  $i_*(\QGr(A/J))$ is closed under submodules and quotients in $\QGr(A)$.

Let $\cM \in \QGr(A/J)$ and suppose that $0 \to \cL \to i_*\cM \to \cN
\to 0$ is an exact sequence in $\QGr(A)$. Let $N$ denote the image and $T$ the cokernel of the map $\omega( i_*\cM \to \cN)$.
Because $\pi\omega \cong \id_{\QGr(A)}$, the map $\pi \omega( i_*\cM \to \cN)$ is an epimorphism.
Hence $\pi T=0$; i.e., $T \in \Fdim(A)$.  

 Because $\pi$ is exact and $T \in \Fdim (A)$,    $\cN \cong \pi N$. 
Now $\cM=\pi' M$ for some $M \in \Gr (A/J)$, so 
$\omega i_* \cM = \omega i_* \pi' M = \omega \pi f_* M \cong 
f_*\omega'\pi' M$. Since $J$ annihilates $f_*\omega'\pi' M$ it annihilates $\omega i_* \cM$. 
But $\omega$ is left exact so $\omega \cL$ is isomorphic to a submodule of $\omega i_*\cM$ and is 
therefore annihilated by $J$. Hence $\omega\cL=f_*L$ for some $L \in \Gr (A/J)$.
Thus, $\cL \cong \pi\omega \cL = \pi f_*L = i_* \pi'L \in
i_*(\QGr(A/J))$. Since $N$ is a quotient of $\omega i_*\cM$ it is also annihilated by $J$; hence we can
apply the argument we just used for $L$ to $N$ and deduce that $\cN \in i_*(\QGr(A/J))$.

That completes the proof of (5) and hence the proof of the theorem (see 
\cite[p.2928]{S} for the definition of a closed immersion).  
\end{proof}

{\bf Remarks.}
(1)
The three functors $(i^*,i_*,i^!)$ provide a ``map''  which we call $i$.
Since $i_*$ is faithful and has a right adjoint, the map $i$ is an affine map in the sense of
Rosenberg \cite[p.278]{R}.  

(2)
It is common in the above situation to define the ``structure sheaves''   of $X=\Projnc(A)$
and $Z=\Projnc(A/J)$  to be $\cO_X:=\pi A$  and $\cO_Z:=\pi'(A/J)$.
The map $i:Z \to X$ then has the property that $i^*\cO_X =\cO_Z$.

(3)
Suppose that $J \subseteq J' \subseteq A$ are graded 2-sided ideals. Let $(f^*,f_*,f^!)$, $(g^*,g_*,g^!)$, and $(h^*,h_*,h^!)$, 
be the adjoint triples associated to the homomorphisms $A \to A/J$, $A/J \to A/J'$, and $A \to A/J'$, respectively.
Consider the diagram
$$
\xymatrix{
\Gr (A/J') \ar[r]^{g_*}  \ar[d]_{\pi''}  & \Gr (A/J) \ar[r]^{f_*}  \ar[d]_{\pi'}  & \Gr (A)  \ar[d]^{\pi}
\\
\QGr(A/J') \ar[r]_{j_*}  & \QGr(A/J) \ar[r]_{i_*}  &  \QGr(A)
}
$$
where $i_*$ and $j_*$ are the unique functors such that $i_*\pi'=\pi f_*$ and $j_*\pi'' = \pi g_*$ and $\pi''$ is the quotient
functor. Since $h_*=f_*g_*$, $i_*j_*\pi'' = \pi h_*$ so $i_*j_*$ is the unique functor $k_*$ such that $\pi h_*=k_*\pi''$.
If we define $k^*=\pi'' h^* \omega$, we can not conclude that $k^*=j^*i^*$. However, since $j^*i^*$ is left adjoint to $i_*j_*=k_*$ we know that $k^* \cong j^*i^*$.

\begin{theorem}
\label{thm2}
Let $J$ be a two-sided graded ideal in a connected graded $k$-algebra  $A$.
Let $X$ be a noetherian scheme having an ample line bundle; for example, let $X$ be 
a quasi-projective scheme over $\Spec(k)$.  
If  there is an equivalence of categories $\Phi:\QGr(A) \to \Qcoh(X)$,  then $\QGr(A/J)  \equiv \Qcoh(Z)$ for some closed subscheme $Z \subseteq X$. 
More precisely, there is a  commutative diagram
$$
\xymatrix{
 \QGr(A/J)  \ar@{-->}[d]_{\equiv} \ar[rr]^{i_*} && \QGr(A) \ar[d]^\Phi
\\ 
\Qcoh(Z) \ar[rr]  && \Qcoh(X)
}
$$
in which the horizontal arrows are the natural inclusion functors and the dotted arrow is an equivalence  
between $ \QGr(A/J) $ and  the essential image of $\Phi i_*$.
\end{theorem}
\begin{proof}
By Theorem \ref{thm.3.2},  $\Projnc(A/J)$ is a closed subspace of $\Projnc(A)$. 
By \cite[Thm. 4.1]{S2}, the closed subspaces of $X$, or $\Qcoh(X)$, are 
the same things as closed subschemes of $X$ in the classical sense.
Hence there is a closed subscheme $Z \subseteq X$
such that   $\QGr(A/J) \equiv \Qcoh(Z)$.  
\end{proof}

{\bf Acknowledgement.}
I thank James Zhang for many conversations over many years about the ideas in this 
corrigendum, and Alex Chirvasitu for questions and conversations that inspired Theorem \ref{thm2}.
I thank the referee for reading the previous draft with care and for identifying some errors and suggesting 
improvements.

\end{document}